\documentclass[11pt]{amsart}    

\usepackage{graphicx}
\usepackage{color}
\usepackage{enumerate}
\usepackage{amsfonts,amsmath,latexsym,amssymb,amsthm}
\usepackage{hyperref}
\usepackage{geometry}
\usepackage{txfonts}
\usepackage{color}
\definecolor{Myred}{cmyk}{0.0,1.0,1.0,0.00}

\newtheorem{claim}{Claim}[section]
\newtheorem{theorem}[claim]{Theorem}
\newtheorem{lemma}[claim]{Lemma}

\begin{document}

\title[Magnetic Dirichlet Laplacian on deformed waveguides]
{Magnetic Dirichlet Laplacian on deformed waveguides}

\author{Daniel Alpay} 
\address{Faculty of Mathematics, Physics, and Computation\\Schmid College of Science and Technology, Chapman University\\ One University Drive Orange, California 92866, USA}
\email{alpay@chapman.edu}
\author{Diana Barseghyan(Schneiderov\'{a})} 
\address{Department of Mathematics, University of Ostrava, 30. dubna 22, 70103 Ostrava, Czech Republic}
\email{diana.schneiderova@osu.cz}
\author{Baruch Schneider} 
\address{Department of Mathematics, University of Ostrava, 30. dubna 22, 70103 Ostrava, Czech Republic}
\email{baruch.schneider@osu.cz}


\maketitle

\begin{abstract}
It is well known that the spectrum of the Dirichlet Laplacian for a two-dimensional waveguide, which is a local deformation of a straight strip, is unstable with respect to waveguide boundary deformations.  This means that, when the waveguide is a straight strip, the spectrum of the Dirichlet Laplacian is purely essential. On the other hand, local boundary perturbations of the straight strip produce eigenvalues below the essential spectrum.
This paper considers the Dirichlet-Laplace operator with a compactly supported magnetic field. Furthermore, we omit the condition that the boundary perturbation is local.  We prove that, in this case, the spectrum of the magnetic Laplacian is stable under small deformations of the waveguide boundary.
\end{abstract}
\bigskip

%%%%%%%%%%%%%%%%%%%%%%%%%%%%%%%%%%%%%%%
\section{Introduction} \label{s: intro}
\setcounter{equation}{0}

The investigation of the dynamics of quantum particles confined to tubular regions yielded interesting results, the most unexpected of which concerned localized states that owed their existence exclusively to the geometry of the confinement
\cite{EK15}. This discovery influenced the much older and more developed theory of electromagnetic waveguides \cite{LCM99} from
the simple reason that, at least in some circumstances, the corresponding Maxwell
equations are well approximated by the same Helmholtz equation encountered in quantum mechanics. 
Now, let us describe in more detail how the existence of bound states depends on the geometry of the region. 
It has long been known that an appropriate bending of a two
dimensional quantum waveguidein an appropriate way induces the existence of bound states (see, for example, references
\cite{ES89}, \cite{GJ92}, and \cite{DE95}). From a mathematical perspective, this implies that the Dirichlet Laplacian on a smooth asymptotically straight, but locally bent, planar waveguide has at least one isolated eigenvalue below the threshold of the essential spectrum. Similar results have been obtained for a locally deformed waveguide, which corresponds to adding a small ``bump'' to the straight waveguide, see \cite{BGRS97} and \cite{BEGK01}. 

We will describe in more detail the particular method of creating the eigenvalues by making the suitable boundary deformation of the straight tube without bending it. Let $f$ be a non-negative, smooth function that vanishes at infinity. We construct
\begin{equation}\label{domain}
\Omega_f = \left\{(x, y) \in \mathbb{R}^2 \,\,\text{such that}\,\,0 < y < \pi +
 \pi f(x)\right\}.
\end{equation}

As mentioned in the cited work \cite{BGRS97}, if $f$ is a smooth, compactly supported function then the essential spectrum of the Dirichlet Laplacian on $\Omega_f$ coincides with the half-line $[1,\infty)$. On the other hand, if $f$ is non-zero, then the discrete spectrum below $1$ is always non-empty. 

The situation changes when a magnetic field is present.
Let $H_{\Omega_f}(\mathcal{A})$ be the Friedrich's extension of the symmetric, semi-bounded operator
\begin{equation}\label{operator}
q_{\Omega_f}^{\mathcal{A}}[\psi]=\|(i\nabla + \mathcal{A})\psi\|^2_{L^2(\Omega_f)},\quad \psi \in \mathcal{H}_0^1(\Omega_f),\end{equation}
where the real-value function $\mathcal{A}$ is a vector potential. Throughout this work, we will suppose that the magnetic vector potential, $\mathcal{A}$, is compactly supported. 
According to the first representation theorem \cite{K66}, there is a unique, nonnegative self-adjoint operator associated with this form:
\begin{equation}\label{1}H_{\Omega_f}(\mathcal{A}) = (i\nabla + \mathcal{A})^2.\end{equation}  

A similar operator appeared in \cite{EK05}. The authors established that if $f$ is a smooth, compactly supported function, then the essential spectrum of $H_{\Omega_f}(\mathcal{A})$ is unchanged and coincides with the half-line $[1,\infty)$. However, if $f$ is small enough, then the discrete spectrum of $H_{\Omega_f}(\mathcal{A})$ is empty due to the presence of the magnetic field.

It should be noted that similar questions regarding the magnetic Dirichlet Laplacian defined on a tube which is a straight tube outside of some compact set were also considered in \cite{KR14}.

The effect of the magnetic field on the discrete spectrum was discussed for different questions in \cite{BBS24}, \cite{BBSZ25} and \cite{BE21}.   

Our work addresses a different problem. We consider the magnetic Dirichlet Laplacian, $H_{\Omega_f}(\mathcal{A})$, on the domain $\Omega_f$, but we do not assume that $f$ must be compactly supported.  We will establish a necessary condition for $f$ that guarantees the stability of the essential spectrum and also the absence of the discrete spectrum due to the magnetic field.
In the conclusion, we discuss the class of perturbations for which we prove that the non-magnetic operator 
$H_{\Omega_f}(0)$ has a non-empty discrete spectrum below the threshold of its essential spectrum. 
 
\section{Main results} \label{s: intro}
\setcounter{equation}{0}

We will prove that the essential spectrum is stable and that there is no discrete spectrum below the threshold of the essential spectrum of the operator (\ref{1}). First, we will address the absence of the discrete spectrum due to a magnetic field. As a second result, we will show that the essential spectrum of the operator (\ref{1}) coincides with the half-line $[1, \infty)$.

\subsection{Absence of the discrete spectrum}

Throughout our work, we assume the following condition on $f$: the first, second, and third derivatives of $f$ exist, as do the first and second derivatives of $|f'|$. Furthermore, 
there exists a positive constant $\alpha$ such that for all $x\in \mathbb{R}$

\begin{equation}\label{(a)}
|f(x)|\le \frac{\alpha}{1+x^2}\,,\quad |f'(x)|\le \frac{\alpha}{1+x^2}\,, \quad |f''(x)|\le \frac{\alpha}{1+x^2}
\,,\quad |f'''(x)|\le \frac{\alpha}{1+x^2}\,.
\end{equation}

\begin{theorem}\label{absence}
Let $B= \mathrm{curl}(\mathcal{A}) \in C_0^1(\mathbb{R}^2)$ be a real-valued magnetic field which is non-trivial
in $\Omega_f$. Suppose the validity of the assumptions (\ref{(a)}). Then there exists $\alpha_0>0$ such that for all $\alpha< \alpha_0$ the discrete spectrum of the operator (\ref{1}) below $[1, \infty)$ is empty. 
\end{theorem}

\begin{proof}

The idea of the proof is based on the rewriting of the quadratic form $q_{\Omega_f}^{\mathcal{A}}$ to another one defined on the straight strip $\Omega_0= \mathbb{R}\times (0, \pi)$. Let us denote 
\begin{equation}\label{g}g(x) = 1 + f(x)\,.
\end{equation}

For any $\psi\in \mathcal{H}^1_0(\Omega_f)$ the form (\ref{operator}) is the following
\begin{equation}\label{first form}
q_{\Omega_f}^{\mathcal{A}}[\psi]= \int_{\Omega_f} \left| i \psi_x(x, y) + a_1(x, y) \psi(x, y)\right|^2
\,d x\,d y\\+ \int_{\Omega_f}  \left|i \psi_y(x, y)+ a_2(x, y)\psi(x, y)\right|^2\,d x\,d y\,,
\end{equation}
where $\mathcal{A}(x, y)= (a_1(x, y), a_2(x, y))$.

We define the transformation $\varphi(x, y)=(g(x))^{1/2}\psi (x, g(x)y)\in \mathcal{H}^1_0(\Omega_0)$. It is easy to check that 

\begin{eqnarray*}
\psi_x(x, g(x)y)= \frac{1}{\sqrt{g(x)}}\varphi_x(x, y)- \frac{g'(x)}{2g(x)\sqrt{g(x)}}\varphi(x, y)- \frac{g'(x) y}{g(x)\sqrt{g(x)}}\varphi_y(x, y)\,,\\
\psi_y(x, g(x)y)=\frac{1}{ g(x) \sqrt{g(x)}} \varphi_y(x, y)\,.
\end{eqnarray*}

Then (\ref{first form}) can be transformed to
\begin{eqnarray*}\nonumber
q_{\Omega_f}^{\mathcal{A}}[\psi]= \int_{\Omega_0} g(x)\left| i \psi_x(x, g(x)y) + a_1(x, g(x)y) \psi(x, g(x)y)\right|^2
\,d x\,d y\\+ \int_{\Omega_0} g(x) \left|i \psi_y(x, g(x)y)+ a_2(x, g(x)y)\psi(x, g(x)y)\right|^2\,d x\,d y\\\nonumber
= \int_{\Omega_0} \left|i \varphi_x(x, y) - \frac{i g'(x)}{2g(x)}\varphi(x, y)- \frac{i g'(x) y}{g(x)} \varphi_y(x, y)+ \tilde{a}_1(x, y)\varphi(x, y)
\right|^2\,d x\,d y\\\nonumber
+ \int_{\Omega_0}\left|\frac{i}{g(x)}\varphi_y(x, y) + \tilde{a}_2(x, y)\varphi(x, y)\right|^2\,d x\,d y\,,
\end{eqnarray*}
where
\begin{equation} \label{modAB}
\tilde a_1(x, y)=  a_1(x, g(x)y),\quad\text{and}\quad \tilde a_2(x, y) = a_2(x, g(x)y).
\end{equation}

In view of the following identity obtained by integration by parts  
$$\int_{\Omega_0}\left(\frac{g'}{g}\right)^2 y (\varphi \overline{\varphi_y}+ \varphi_y \overline{\varphi})\,d x\,d y=- \int_{\Omega_0}\left(\frac{g'}{g}\right)^2 |\varphi|^2\,d x\,d y\,,$$
we get  
\begin{eqnarray} \nonumber
q_{\Omega_f}^{\mathcal{A}}[\psi]
= \int_{\Omega_0} \bigg( \left| -i \varphi_x + \tilde a_1 \varphi
\right|^2 + \left| -i \varphi_y + \tilde a_2 \varphi
\right|^2 
 -\frac{g'}{2g} (\overline{\varphi_x} \varphi+ \varphi_x \overline{\varphi})\\\nonumber -
\frac{g'}{g} y (\overline{\varphi_x} \varphi_y+ \varphi_x \overline{\varphi_y})+ i \frac{\tilde a_1 g'}{g} y
(\varphi_y \overline{\varphi}- \overline{\varphi_y} \varphi)
- \left(\frac{g'}{2g}\right)^2 |\varphi|^2 \\\nonumber+ \left(\frac{g'}{g}\right)^2 y^2 |\varphi_y|^2-
\left(1- \frac{1}{g^2}\right) |\varphi_y|^2+ i \tilde a_2\left(1- \frac{1}{g}\right) (\varphi_y \overline{\varphi}- \overline{\varphi_y} \varphi)
\biggr)\,d x\,d y\\\label{q}= \int_{\Omega_0} \left| i \nabla \varphi + \tilde A \varphi\right|^2 \,d x\,d y- I(\varphi)\,, 
\end{eqnarray}
where $\tilde A(x, y)= (\tilde a_1(x, y), \tilde a_2(x, y))$ and 
\begin{eqnarray*}
I(\varphi)=\int_{\Omega_0} \bigg(\frac{g'}{2g} (\overline{\varphi_x} \varphi+ \varphi_x \overline{\varphi}) +
\frac{g'}{g} y (\overline{\varphi_x} \varphi_y+ \varphi_x \overline{\varphi_y})- i \frac{\tilde a_1 g'}{g} y
(\overline{\varphi_y} \varphi-\varphi_y \overline{\varphi})
+ \left(\frac{g'}{2g}\right)^2 |\varphi|^2 \\\nonumber- \left(\frac{g'}{g}\right)^2 y^2 |\varphi_y|^2+
\left(1- \frac{1}{g^2}\right) |\varphi_y|^2+ i \tilde a_2 \left(1- \frac{1}{g}\right) (\varphi_y \overline{\varphi}- \overline{\varphi_y} \varphi)
\biggr)\,d x\,d y\,.
\end{eqnarray*}

Before to continue the proof need the following lemma (the proof of which we present in the Appendix):

\begin{lemma}\label{app.}
Let $\hat B \in C_0^1(\mathbb{R}^2)$ be a real-valued magnetic field compactly supported in $\Omega_0$. Then one can choose the magnetic potential $A= (\hat a_1, \hat a_2)$ corresponding to $\hat B$ in such a way that $\hat a_1,\, \hat a_2\in L^\infty(\Omega_0)$. 
\end{lemma}

Let us choose $\tilde a_1$ and $\tilde a_2$ satisfying the properties described in Lemma \ref{app.}. Correspondingly the original magnetic potential will be obtained by relations (\ref{modAB}).
In view of this and using expression (\ref{g})  it is easy to check that $I(\varphi)$ can be estimated as follows
\begin{eqnarray}\nonumber
|I(\varphi)|\le \frac{(1+ 2\pi)}{2\|g\|_\infty} \int_{\Omega_0}|f'| |\varphi_x|^2\,d x\,d y+  \frac{\pi}{\|g\|_\infty}  \int_{\Omega_0} |f'|\left(1+ \|\tilde a_1\|_\infty+\frac{|f'| \pi}{\|g\|_\infty}\right) |\varphi_y|^2\,d x\,d y\\\nonumber+ \int_{\Omega_0}\left(2f -\frac{f^2(3+ 2f)}{\|g\|_\infty^2}
+\frac{\|\tilde a_2\|_\infty \cdot f}{\|g\|_\infty}\right)|\varphi_y|^2\,d x\,d y+
 \frac{1}{2\|g\|_\infty} \int_{\Omega_0}|f'| \left(1 +2\|\tilde a_1\|_\infty \pi+ \frac{|f'|}{2\|g\|_\infty}\right)|\varphi|^2\,d x\,d y\\\nonumber+\frac{1}{\|g\|_\infty} \int_{\Omega_0} \|\tilde a_2\|_\infty\cdot f\cdot |\varphi|^2\,d x\,d y\\
 \label{I}\le \int_{\Omega_0} G_f^1 (|\varphi_x|^2+ |\varphi_y|^2)\,d x\,d y+ \int_{\Omega_0} G_f^2 |\varphi|^2\,d x\,d y\,,
\end{eqnarray}
where $\|\cdot\|\equiv \|\cdot\|_{L^\infty(\mathbb{R})}$ or $\|\cdot\|\equiv \|\cdot\|_{L^\infty(\mathbb{R}^2)}$ and 
\begin{eqnarray}\label{Gf1}
G_f^1= \frac{(1+ 2\pi)}{2\|g\|_\infty}|f'|+ \frac{\pi |f'|}{\|g\|_\infty}\left(1+ \|\tilde a_1\|_\infty +\frac{|f'| \pi}{\|g\|_\infty}\right)+ 2f- \frac{ f^2(3+ 2f)}{\|g\|_\infty^2}+ \frac{\|\tilde a_2\|_\infty f}{\|g\|_\infty}\,,\\\label{Gf}
G_f^2= \frac{1}{2\|g\|_\infty}|f'| \left(1+ 2\|\tilde a_1\|_\infty \pi+ \frac{|f'|}{2\|g\|_\infty}\right)+\frac{\|\tilde a_2\|_\infty f}{\|g\|_\infty}\,.
\end{eqnarray}

Using the following pointwise inequality 
$$
\left|\nabla \varphi\right|^2\le 2\left|i\nabla \varphi+ \tilde{A}\varphi\right|^2 + 2\left(\tilde a_1^2+ \tilde a_2^2\right)
|\varphi|^2\,, \quad \varphi\in \mathcal{H}^1(\Omega_0)\,,
$$
the integral $\int_{\Omega_0}G_f^1 (|\varphi_x|^2+ |\varphi_y|^2)\,d x\,d y$ can be estimated as follows

$$
\int_{\Omega_0}G_f^1 (|\varphi_x|^2+ |\varphi_y|^2)\,d x\,d y\le \int_{\Omega_0} 2 G_f^1 (|i \nabla \varphi+ \tilde A \varphi|^2+ (\|\tilde a_1\|_\infty^2+ \|\tilde a_2\|_\infty^2)|\varphi|^2)\,d x\,d y\,.
$$

Finally using (\ref{q}), (\ref{I}) and the above estimate we have

\begin{equation}\label{fin.}
q_{\Omega_f}^{\mathcal{A}}[\psi]\ge \int_{\Omega_0} (1- 2 G_f^1 )\left| i \nabla \varphi + \tilde A \varphi\right|^2 \,d x\,d y-  2(\|\tilde a_1\|_\infty^2+ \|\tilde a_2\|_\infty^2) \int_{\Omega_0}  G_f^1  |\varphi|^2\,d x\,d y-
 \int_{\Omega_0}  G_f^2  |\varphi|^2\,d x\,d y
\,.
\end{equation}

For the further proof we need the following Hardy-type inequality established in work \cite{BBSZ25}.
\begin{lemma}\label{lemma1}
For any function $g\in \mathcal{H}_0^1(\Omega_0)$, the following estimate holds    
 \begin{eqnarray*}
\int_{\Omega_0}\left(h^2 \left|i \nabla g+ \tilde{A} g\right|^2- h^2|g|^2\right)\,d x\,d y\ge C_{\tilde{A}}\int_{\Omega_0}\frac{h^2}{1+x^2} |g|^2\,d x\,d y+ \int_{\Omega_0}h'' |g|^2\,d x\,d y\,,
\end{eqnarray*}
where $C_{\tilde{A}}>0$ is a constant depending on the magnetic field and $h= h(x):\mathbb{R}\to (0, \infty)$ is a smooth function. 
\end{lemma}

In view of assumptions of Theorem \ref{absence} function $h= \sqrt{1- 2 G_1^f}$ has first and second derivatives.
Next, using (\ref{fin.}) and the above lemma applied with $h$ one gets
\begin{eqnarray}\nonumber
q_{\Omega_f}^{\mathcal{A}}[\psi]- \|\psi\|^2_{L^2(\Omega_f)}\\\nonumber\ge \int_{\Omega_0} (1- 2 G_1^f) |i \nabla \varphi+ \tilde A \varphi|^2\,d x\,d y-
\int_{\Omega_0} 2 G_f^1(\|\tilde a_1\|_\infty^2+ \|\tilde a_2\|_\infty^2)|\varphi|^2\,d x\,d y
- \int_{\Omega_0} G_f^2 |\varphi|^2\,d x\,d y- \int_{\Omega_0} |\varphi|^2\,d x\,d y\\\nonumber
\ge C_{\tilde{A}}\int_{\Omega_0}\frac{ (1- 2G_f^1)}{1+x^2} |\varphi|^2\,d x\,d y+ \int_{\Omega_0} ((1- 2G_f^1)^{1/2})'' |\varphi|^2\,d x\,d y- \int_{\Omega_0} 2G_f^1 |\varphi|^2\,d x\,d y \\\label{last}-
(\|\tilde a_1\|_\infty^2+ \|\tilde a_2\|_\infty^2) \int_{\Omega_0} 2G_f^1|\varphi|^2\,d x\,d y- \int_{\Omega_0} G_f^2 |\varphi|^2\,d x\,d y\,.
\end{eqnarray}

In view of the expressions (\ref{Gf1})- (\ref{Gf}) and assumptions (\ref{(a)}) one can to show that 
 $$
 |G_f^1(x)|,\, |(G_f^1)'(x)|,\,|(G_f^1)''(x)|,\, |G_f^2(x)| \le \frac{C \alpha}{1+ x^2}\,, x\in \mathbb{R}\,,
$$
where the constant $C$ depends only on $\|g\|_\infty,\, \|\tilde a_1\|_\infty,\, \|\tilde a_1\|_\infty$.

Hence one can choose $\alpha$ small enough to guarantee that the expression
 
\begin{eqnarray*}\label{final}
-C_{\tilde{A}}\frac{2 G_f^1}{1+x^2} + ((1- 2G_f^1)^{1/2})''- 2 G_f^1-
2 G_f^1(\|\tilde a_1\|_\infty^2+ \|\tilde a_2\|_\infty^2)-  G_f^2\\\nonumber= -C_{\tilde{A}}\frac{2 G_f^1}{1+x^2} - \frac{1}{(1- 2G_f^1)^{3/2}}
\left((G_f^1)'' (1- 2 G_f^1)+ ((G_f^1)')^2\right)- 2G_f^1-
2G_f^1(\|\tilde a_1\|_\infty^2+ \|\tilde a_2\|_\infty^2)-  G_f^2
\end{eqnarray*}
will not exceed   $\frac{C_{\tilde{A}}}{1+ x^2}$,  which together with (\ref{last}) means that 
$$
q_{\Omega_f}^{\mathcal{A}}[\psi]- \|\psi\|^2_{L^2(\Omega_f)}\ge 0\,.
$$
The above inequality establishes our proof.

\end{proof}

\subsection{Stability of the essential spectrum}

In this subsection we establish that the compactly supported magnetic field cannot change the essential spectrum as it has been shown in the following theorem:

\begin{theorem}\label{stability}
Suppose the assumptions of Theorem \ref{absence} but this time allowing the magnetic field to be zero. 
Then the essential spectrum of the operator (\ref{1}) coincides with the half-line $[1, \infty)$.
\end{theorem}

\begin{proof}

To prove that any non-negative number $\mu\ge 1$ belongs to the essential
spectrum of $H_{\Omega_f}(\mathcal{A})$, we will use Weyl's criterion
\cite[Thm.~VII.12]{RS81}: we have to find a sequence
$\{\varphi_n\}_{n=1}^\infty\subset D(H_{\Omega_f}(\mathcal{A}))$ of unit vectors,
$\|\varphi_n\|=1$, which converges weakly to zero and
\begin{equation}\label{W}
\|H_{\Omega_f}(\mathcal{A})\varphi_n-\mu\varphi_n\|_{L^2(\Omega_f)}\to 0 \qquad\text{as}\quad n\to\infty
\end{equation}% ------------- %
holds. 
% ------------- %

For each $\mu= 1+ k^2,\,k\in\mathbb{Z}$ we will construct the sequence $\varphi_n\in C_0^\infty(\Omega_f)$ as follows
\begin{equation}\label{psi}
\varphi_n(x, y)= \sqrt{\frac{1}{n}}h\left(\frac{x}{n}\right) e^{i k x} \sin\left(\frac{y}{g(x)}\right),\quad n\in \mathbb{N},
\end{equation}
where $h\in C_0^\infty(\mathbb{R})$ is a smooth function with support in the interval $(1, 2)$ with $L^2$ norm equal to one and $g$ is defined by (\ref{g}).

We have 
\begin{eqnarray*}
(\varphi_n)_{x x}(x, y)= \frac{1}{\sqrt{n}}\biggl(\frac{1}{n^2}h''\left(\frac{x}{n}\right)\sin\left(\frac{y}{g(x)}\right)- \frac{2}{n}\frac{y g'(x)}{g(x)^2}h'\left(\frac{x}{n}\right) \cos\left(\frac{y}{g(x)}\right)\\- \frac{y^2 g'(x)^2}{g(x)^4} h\left(\frac{x}{n}\right) \sin\left(\frac{y}{g(x)}\right)- \frac{(y g''(x) g(x)-2 y g'(x)^2 )}{g(x)^3} h\left(\frac{x}{n}\right) \cos\left(\frac{y}{g(x)}\right)\\+ \frac{2i k}{n} h'\left(\frac{x}{n}\right) \sin\left(\frac{y}{g(x)}\right)- \frac{2i k y g'(x)}{g(x)^2} h\left(\frac{x}{n}\right) \cos\left(\frac{y}{g(x)}\right)- k^2 h\left(\frac{x}{n}\right) \sin\left(\frac{y}{g(x)}\right)\biggr) e^{i k x}\,,
\\ (\varphi_n)_{y y}(x, y)=- \frac{1}{g^2(x) \sqrt{n}} h\left(\frac{x}{n}\right) \sin\left(\frac{y}{g(x)}\right) e^{i k x}\,.
\end{eqnarray*}

Since the support of magnetic potential is a compact set then for large values of $n$ operator $H_{\Omega_f}(\mathcal{A})(\varphi_n)$ coincides with the Laplace operator applied at $\varphi_n$. Then using the assumptions (\ref{(a)}) and the above expressions one gets

\begin{eqnarray*}\nonumber
\int_{\Omega_f}\left|H_{\Omega_f}(\mathcal{A})\varphi_n- (k^2+1) \varphi_n\right|^2\,d x\,d y=
\int_{\Omega_f}\left|-\Delta \varphi_n- (k^2+1)\varphi_n\right|^2\,d x\,d y\\\nonumber=
\frac{1}{n} \int_n^{2n} \int_0^{\pi g(x)} \biggl|-\frac{1}{n^2}h''\left(\frac{x}{n}\right)\sin\left(\frac{y}{g(x)}\right)+ \frac{2}{n}\frac{y g'(x)}{g(x)^2}h'\left(\frac{x}{n}\right) \cos\left(\frac{y}{g(x)}\right)\\+ \frac{y^2 g'(x)^2}{g(x)^4} h\left(\frac{x}{n}\right) \sin\left(\frac{y}{g(x)}\right)+ \frac{(y g''(x) g(x)-2 y g'(x)^2 )}{g(x)^3} h\left(\frac{x}{n}\right) \cos\left(\frac{y}{g(x)}\right)\\- \frac{2i k}{n} h'\left(\frac{x}{n}\right) \sin\left(\frac{y}{g(x)}\right)+ \frac{2i k y g'(x)}{g(x)^2} h\left(\frac{x}{n}\right) \cos\left(\frac{y}{g(x)}\right)\\\nonumber+ k^2 h\left(\frac{x}{n}\right) \sin\left(\frac{y}{g(x)}\right) +\frac{1}{g(x)^2} h\left(\frac{x}{n}\right) \sin\left(\frac{y}{g(x)}\right)- (k^2+1)h\left(\frac{x}{n}\right) \sin\left(\frac{y}{g(x)}\right)\biggr|^2\,d x\,d y
\\\le 2\int_1^2 \int_0^\pi g(n x)\biggl| -\frac{1}{n^2}h''(x)\sin y+ \frac{2}{n}\frac{y g'(n x)}{g(n x)}h'(x)\cos y\\+ \frac{y^2 g'(n x)^2}{g(n x)^2} h(x) \sin y+ \frac{(y g''(n x) g(n x)-2 y g'(n x)^2 )}{g(n x)^2} h(x) \cos y\\- \frac{2i k}{n} h'(x) \sin y+ \frac{2i k y g'(n x)}{g(n x)} h(x) \cos y\biggr|^2\,d x\,d y\\+ 2 \int_1^2 \int_0^\pi g(n x) \left(1-\frac{1}{g(n x)^2}\right)^2 h(x)^2 \sin^2y\,d x\,d y\,.
\end{eqnarray*}

In view of the fact that $g'(x)= f(x)$ and $1-\frac{1}{g(n x)^2}= \frac{2 f(n x)+ f(n x)^2}{(1+ f(n x))^2}$ and assumptions (\ref{(a)}) it is easy to check that for large values of $n$ the right-hand side of the above estimate has the asymptotic behaviour $\mathcal{O}\left(\frac{1}{n^2}\right)$. 

Let us now estimate the $L^2$ norm of $\varphi_n$. In view of assumptions (\ref{(a)}) for large values of $n$ one has
\begin{eqnarray*}
\int_{\Omega_f}|\varphi_n|^2\,d x\,d y=\frac{1}{n} \int_n^{2n}\int_0^{\pi g(x)}\left|h\left(\frac{x}{n}\right)\right|^2 \left(\sin\left(\frac{y}{g(x)}\right)\right)^2\,d x\,d y\\=
\frac{1}{n} \int_n^{2n}\int_0^{\pi}g(x) \left|h\left(\frac{x}{n}\right)\right|^2 \sin^2y\,d x\,d y=
\\=
\int_1^2\int_0^{\pi}g(n x)|h(x)|^2 \sin^2y\,d x\,d y= \frac{\pi}{2} \int_1^2 g(n x)|h(x)|^2\,d x\\=
\frac{\pi}{2} \int_1^2 (1+ f(n x))|h(x)|^2\,d x= \frac{\pi}{2}(1+ o(1))\,.
\end{eqnarray*}

Hence the Weyl sequence satisfying (\ref{W}) we can choose $\frac{1}{\|\varphi_n\|_{L^2(\Omega_f)}}\varphi_n$.
This finishes the proof.

\end{proof}

\subsection{Non-emptiness of the discrete spectrum for non-magnetic Laplacian}
\label{non-empty}

In this section we discuss about the non-emptiness of the discrete spectrum of the non-magnetic Laplacian $H_{\Omega_f}(0)$.
In \cite{BGRS97} it was proven that in case if $f$ is a smooth compactly supported non-zero function then the discrete spectrum of $H_{\Omega_f}(0)$ below $1$ is always non-empty.

We will show the non-emptiness of the discrete spectrum for the class of perturbations without assuming for them to be compactly supported.

\begin{theorem}
Suppose the assumptions of Theorem \ref{stability} with the magnetic field be absent. Let us assume that the perturbation $f$ satisfies the following assumption
$$
|f'(x)|\le \frac{2 \sqrt{3}}{\sqrt{4\pi^2+ 3}}\sqrt{f(x)(2+ f(x))}\,,
$$
and there is at least one point where the above inequality is strict.
Then the discrete spectrum of $H_{\Omega_f}(0)$ below the threshold of the essential spectrum is non-empty.
 \end{theorem}

\begin{proof}
Let us consider function
$$
\varphi(x, y)= r(x) \sin\left(\frac{y}{g(x)}\right),
$$
where as usual $g$ is given by (\ref{g}) and $r$ is a smooth real-valued function to be chosen later. 
We are going to show that 
\begin{equation}\label{aim}
q_{\Omega_f}^0[\varphi]- \|\varphi\|_{L^2(\Omega_f)}^2< 0\,.
\end{equation}

We have

\begin{eqnarray*}
\varphi_x(x, y)= r'(x) \sin\left(\frac{y}{g(x)}\right)- \frac{g'(x) y}{g(x)^2} r(x) \cos\left(\frac{y}{g(x)}\right)\,,\\
\varphi_y(x, y)= \frac{1}{g(x)} r(x) \cos\left(\frac{y}{g(x)}\right)\,.
\end{eqnarray*}

Hence 

\begin{eqnarray*}\nonumber
q_{\Omega_f}^0[\varphi]- \|\varphi\|_{L^2(\Omega_f)}^2\\=\int_{\Omega_f}\left| r'(x) \sin\left(\frac{y}{g(x)}\right)- \frac{g'(x) y}{g(x)^2} r(x) \cos\left(\frac{y}{g(x)}\right)\right|^2\,d x\,d y
+ \int_{\Omega_f}\left| \frac{1}{g(x)} r(x) \cos\left(\frac{y}{g(x)}\right)\right|^2\,d x\,d y\\-  \int_{\Omega_f} r(x)^2 \sin^2\left(\frac{y}{g(x)}\right)\,d x\,d y\\\nonumber=
\int_{\Omega_f}\biggl(r'(x)^2 \sin^2\left(\frac{y}{g(x)}\right)- \frac{g'(x) y}{g(x)^2} r(x) r'(x) \sin\left(\frac{2y}{g(x)}\right)
\\\label{last1}+ \frac{g'(x)^2 y^2}{g(x)^4} r(x)^2 \cos^2\left(\frac{y}{g(x)}\right)
+ \frac{ r(x)^2}{g(x)^2} \cos^2\left(\frac{y}{g(x)}\right) -r(x)^2 \sin^2\left(\frac{y}{g(x)}\right)\biggr)\,d x\,d y\\
= \int_{\mathbb{R}} \int_0^{\pi g(x)}\biggl(r'(x)^2 \sin^2\left(\frac{y}{g(x)}\right)- \frac{g'(x) y}{g(x)^2} r(x) r'(x) \sin\left(\frac{2y}{g(x)}\right)
\\+ \frac{g'(x)^2 y^2}{g(x)^4} r(x)^2 \cos^2\left(\frac{y}{g(x)}\right)
+ \frac{ r(x)^2}{g(x)^2} \cos^2\left(\frac{y}{g(x)}\right) -r(x)^2 \sin^2\left(\frac{y}{g(x)}\right)\biggr)\,d x\,d y\\
=  \int_{\mathbb{R}} \int_0^{\pi} g(x) \biggl(r'(x)^2 \sin^2y- \frac{g'(x) y}{g(x)} r(x) r'(x) \sin(2 y)
\\+ \frac{g'(x)^2 y^2}{g(x)^2} r(x)^2 \cos^2y
+ \frac{ r(x)^2}{g(x)^2} \cos^2y -r(x)^2 \sin^2y\biggr)\,d x\,d y
\,.
\end{eqnarray*}

From the straightforward calculations and the integration by parts the above expression performs
 
\begin{eqnarray*}
q_{\Omega_f}^0[\varphi]- \|\varphi\|_{L^2(\Omega_f)}^2\\= \frac{\pi}{2} \int_{\mathbb{R}}\left( r'(x)^2 g(x)+ r(x)^2 \left(-\frac{g''(x)}{2}+ \frac{g'(x)^2 (2\pi^2+ 3)}{6g(x)}+ \frac{1}{g(x)}- g(x)\right)\right)\,d x\,d y\,.
\end{eqnarray*}

Finally with $v(x):= r(x) \sqrt{g(x)}$ we rewrite 

\begin{eqnarray}\nonumber
q_{\Omega_f}^0[\varphi]- \|\varphi\|_{L^2(\Omega_f)}^2\\\label{latter}= \frac{\pi}{2} \left(\int_{\mathbb{R}} v'(x)^2\,d x+ \int_{\mathbb{R}}v(x)^2 \left(\frac{g'(x)^2 (4\pi^2+ 3)}{12g(x)^2}+ \frac{1}{g(x)^2}- 1\right)\,d x\right)\,.
\end{eqnarray}

In view of (\ref{g})  
$$
\frac{g'(x)^2 (4\pi^2+ 3)}{12g(x)^2}+ \frac{1}{g(x)^2}- 1= \frac{f'(x)^2 (4\pi^2+ 3)}{12(1+ f(x))^2}- \frac{2f(x)+ f(x)^2}{(1+ f(x))^2}\,. 
$$

The latter is non-positive and non-trivial due to the assumptions of the theorem and then thanks to work \cite{S76} the operator $-\frac{d^2}{d x^2} +\left(\frac{g'(x)^2 (4\pi^2+ 3)}{12g(x)^2}+ \frac{1}{g(x)^2}- 1\right)$ has at least one negative bound state. Let us choose $v(x)$ be the eigenfunction corresponding to this negative eigenvalue, then by virtue of (\ref{latter}) we obtain the validity of (\ref{aim}).

\end{proof}

\section{Appendix}
\label{appendix}
 
Let us construct 
\begin{eqnarray*}
\hat a_1 (x, y) = - y \int_0^1 \hat B(u x, u y )\, u\,d u,\\ \hat a_2 (x, y) = x \int_0^1 \hat B(u x, u y)\, u\, d u. 
\end{eqnarray*}

One can check that indeed the magnetic potential $\hat A= (\hat a_1, \hat a_2)$ corresponds to magnetic field $\hat B$ as follows
\begin{eqnarray*}
(\hat a_2)_x - (\hat a_1)_y= \int_0^1 \hat B(u x, u y )\, u\,d u+ x  \int_0^1 (\hat B(u x, u y ))_x\, u\,d u+
\int_0^1 \hat B(u x, u y )\, u\,d u\\+  y \int_0^1 (\hat B(u x, u y ))_y\, u\,d u
=2 \int_0^1 \hat B(u x, u y )\, u\,d u+ \int_0^1 (\hat B(u x, u y ))_u\, u^2\,d u=
\hat B(x, y)\,.
\end{eqnarray*}

In view of the fact that the support of the magnetic field is a compact set one can notice also the boundedness of $\hat a_1$ and $\hat a_2$.

%%%%%%%%%%%%%%%%%%%%%%%%%%%%%%%%%%%%%%%%%%%%%%
\subsection*{Acknowledgements}
D.A. thanks the Foster G. and Mary McGaw Professorship in Mathematical Sciences, which supported this research. 

\bigskip

\subsection*{Conflict of Interest}
The authors declare that they have no competing interests related to the publication of this paper.

\subsection*{Authors' Contributions}
All authors have contributed equally to the manuscript and have drafted, read, and approved the final version of the manuscript, which is the result of an intensive collaboration.


\begin{thebibliography}{10}

\bibitem{BBS24} J.~Bory-Reyes, D.~Barseghyan, B.~ Schneider, Magnetic Schr\"{o}dinger operator with the potential supported in a curved two-dimensional strip,  Mediterranean Journal of Mathematics 21(3) (2024), 1--15. 

\bibitem{BBSZ25} D.~ Barseghyan, S.~ Bernstein, B.~ Schneider
and M.~ Zimmermann, Magnetic Dirichlet Laplacian in curved waveguides, Opuscula Math. 45, no. 3 (2025), 293--305.

\bibitem{BE21} D.~Barseghyan, P.~Exner, Magnetic field influence on the discrete spectrum of locally deformed leaky wires, Reports on Mathematical Physics 88 (1) (2015), 47--57.

\bibitem{BEGK01} D.~ Borisov, P. ~Exner, R.R. ~Gadyl’shin and D.~Krej\v ci\v
 r\'{\i}k, Bound states in weakly deformed strips and layers, Ann. Henri  Poincar\'e 2 (2001), 553--572.

\bibitem{BGRS97} W.~Bulla, F.~Gesztesy, W.~Renger and B.~Simon,
Weakly coupled bound states in quantum waveguides,
Proc. Amer. Math. Soc. 125 (1997),  no. 5, 1487--1495.

\bibitem{DE95} P. ~Duclos and P. ~Exner, Curvature-induced bound states in quantum waveguides in two and three dimensions, Rev. Math. Phys. 7 (1995), 73--102.

\bibitem{EK05} T.~Ekholm, H.~Kova\v{r}\'{\i}k, Stability of the Magnetic Schr\"{o}dinger Operator in a Waveguide,
Communications in Partial Differential Equations 30 (2005), 539--565.

\bibitem{EK15} P.~Exner, H.~Kova\v{r}\'{\i}k, Quantum Waveguides, Springer International, Heidelberg 2015.

\bibitem{ES89} P.~Exner, P.~\v{S}eba,
Bound states in curved quantum wavequides, J. Math. Phys. 30 (1989), 2574--2580.

\bibitem{GJ92} J. ~Goldstone and R.L.~ Jaffe, Bound states in twisting tubes, Phys. Rev. B45 (1992), 14100--14107.

\bibitem{K66} T.~ Kato, Perturbation Theory for Linear Operators, Springer, Berlin-Heidelberg-New York, 1966.
  
\bibitem{KR14} D.~ Krej\v{c}i\v{r}\'{\i}k, N. ~Raymond, Magnetic effects in curved quantum waveguides,
Ann. Henri Poincare 15 (2014), 1993--2024.
 
\bibitem{LCM99} J.T.~ Londergan, J.P.~ Carini, D.P.~ Murdock: Binding and Scattering in Two-Dimensional
Systems. Applications to Quantum Wires, Waveguides and Photonic Crystals, Springer
LNP m60, Berlin 1999.  
  
\bibitem{RS81} M.~ Reed, B. ~Simon, Methods of Modern Mathematical Physics, I. Functional Analysis, II. Fourier Analysis, IV. Analysis of Operators. Self-Adjointness, Academic Press, New York 1981, 1975, 1978.


\bibitem{S76} B.~Simon, The Bound State of Weakly Coupled Schr\"{o}dinger Operators in
One and Two Dimensions, Ann. of Physics 97 (1976) 279--288.
\end{thebibliography}
\end{document}